\documentclass[a4paper,cleveref,english]{lipics-v2021}
\usepackage[utf8]{inputenc}

\usepackage{xspace}
\usepackage{csquotes}

\usepackage[table]{xcolor}

\newcolumntype{W}{>{\columncolor{gray!10}\centering\arraybackslash}X}
\newcolumntype{Y}{>{\columncolor{gray!20}\centering\arraybackslash}X}
\newcolumntype{Z}{>{\columncolor{gray!30}\centering\arraybackslash}X}

\newcommand{\invref}[1]{\hyperref[#1]{(I\ref*{#1})}\xspace}

\newcommand{\tw}{\operatorname{tw}}
\newcommand{\dist}{\operatorname{dist}}

\title{Clustered independence and bounded treewidth}

\author{Kolja Knauer}{School of Mathematical Sciences, Hebei Key Laboratory of Computational Mathematics and Applications, Hebei Normal University, Shijiazhuang 050024, P. R. China, Departament de Matem\'{a}tiques i Inform\'{a}tica, Universitat de Barcelona and Centre de Recerca Matemàtica, Barcelona, Spain.}{kolja.knauer@ub.edu}{https://orcid.org/0000-0002-8151-2184}{The grant of The Natural Science Foundation of Hebei Province (project No. A2023205045),
 PID2022-137283NB-C22 funded by MICIU/AEI/10.13039/501100011033 and ERDF/EU, Severo Ochoa and María de Maeztu Program for Centers and Units of Excellence in R\&D (CEX2020-001084-M), ANR project MIMETIQUE: ANR-25-CE48-4089-01.}

\author{Torsten Ueckerdt}{Karlsruhe Institute of Technology}{torsten.ueckerdt@kit.edu}{https://orcid.org/0000-0002-0645-9715}{Funded by the Deutsche Forschungsgemeinschaft (DFG, German Research Foundation) – 520723789}

\authorrunning{K. Knauer and T. Ueckerdt}

\Copyright{K. Knauer and T. Ueckerdt}

\ccsdesc[100]{Mathematics of computing $\rightarrow$ Discrete mathematics $\rightarrow$ Combinatorics $\rightarrow$ Combinatoric problems}

\keywords{treewidth, clustered sets}

\hideLIPIcs
\nolinenumbers

\begin{document}

\maketitle

\begin{abstract}
 A set $S\subseteq V$ of vertices of a graph $G$ is a \emph{$c$-clustered set} if it induces a subgraph with components of order at most $c$ each, and $\alpha_c(G)$ denotes the size of a largest $c$-clustered set. 
    For any graph $G$ on $n$ vertices and treewidth $k$, we show that $\alpha_c(G) \geq \frac{c}{c+k+1}n$, which improves a result of Dvo\v{r}{\'a}k and Wood [Innov.\ Graph Theory, 2025], while we construct $n$-vertex graphs $G$ of treewidth~$k$ with $\alpha_c(G)\leq \frac{c}{c+k}n$.
    In the case $c\leq 2$ or $k=1$ we prove the better lower bound $\alpha_c(G) \geq \frac{c}{c+k}n$, which settles a conjecture of Chappell and Pelsmajer [Electron.\ J.\ Comb., 2013] and is best-possible.
    Finally, in the case $c=3$ and $k=2$, we show $\alpha_c(G) \geq \frac{5}{9}n$ which is best-possible.
\end{abstract}

\section{Introduction}

Let $G = (V,E)$ be a graph and $c$ a positive integer.
We call a subset $S \subseteq V$ of vertices of $G$ a \emph{$c$-clustered set} if every connected component of the subgraph $G[S]$ of $G$ induced by $S$ has at most $c$ vertices.
We define the \emph{$c$-clustered independence number} $\alpha_c(G)$ of $G$ as the maximum size of a $c$-clustered set in $G$.
In particular, the $1$-clustered sets of $G$ are exactly the independent sets of $G$ and $\alpha_1(G)$ equals the independence number $\alpha(G)$.
On the other hand, for each $2$-clustered set $S$ of $G$ the subgraph $G[S]$ is a collection of vertices and edges in $G$ with no edge of $G$ between these components, and $\alpha_2(G)$ is the largest number of vertices in $G$ inducing only isolated vertices and edges.

\begin{figure}
    \centering
    \includegraphics{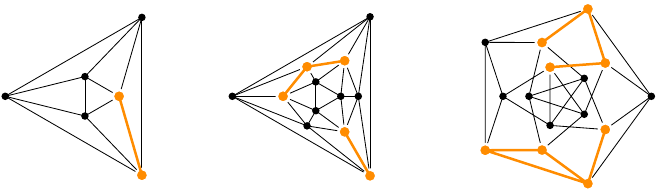}
    \caption{Examples of a $2$-clustered set (left), a $3$-clustered set (middle), and a $4$-clustered set (right).}
    \label{fig:example-clustered-sets}
\end{figure}

In the literature, $c$-clustered sets appear primarily as the color classes of $c$-clustered colorings.
A $t$-coloring $\phi \colon V \to [t]$ of the vertices of a graph $G=(V,E)$ is a \emph{$c$-clustered coloring} if there is no monochromatic connected subgraph on more than $c$ vertices in $G$.
In other words, each color class $\phi^{-1}(i)$, $i=1,\ldots,t$, is a $c$-clustered set.
The \emph{$c$-clustered chromatic number} $\chi_c(G)$ is then the minimum $t$ for which $G$ admits a $c$-clustered $t$-coloring.
So, for example, $\chi_1(G)$ is equal to the classical chromatic number $\chi(G)$, while $\chi_2(G)$ is the smallest number of $2$-clustered sets into which the vertex set $V(G)$ can be partitioned.
Clearly, for any $c \geq 1$ and any graph $G$ we have
\begin{equation}
    \alpha_c(G) \geq \frac{|V(G)|}{\chi_c(G)}, \quad \text{ or equivalently } \quad \frac{\alpha_c(G)}{|V(G)|} \geq \frac{1}{\chi_c(G)}.\label{eq:trivial-LB}
\end{equation}

In this paper, we focus on the quantity $\alpha_c(G) / |V(G)|$, i.e., the proportion of vertices of $G$ that we can put into a $c$-clustered set, and seek to find (for graphs of a particular class) universal lower bounds that significantly improve on the $1/\chi_c(G)$ in~\cref{eq:trivial-LB}.
We shall focus on graphs of treewidth $k$, as  introduced by Bertelè and Brioschi~\cite{BB72}, rediscovered by Halin~\cite{Hal76}, and again rediscovered by Robertson and Seymour~\cite{RS86}. We also briefly discuss other graph classes at the very end.

\subparagraph*{Our Results.}

We are interested in the largest $c$-clustered independence number guaranteed in each $n$-vertex graph $G$ with $\tw(G) \leq k$, i.e., graph of treewidth at most $k$.
To this end, let us define
\begin{equation}
    x_{k,c} = \inf\{ \frac{\alpha_c(G)}{|V(G)|} \colon \tw(G) \leq k\}.\label{eq:definition-x-kc}
\end{equation}

In fact, for lower bounds $\ell_{k,c}\leq x_{k,c}$ we will show the slightly stronger statement that every graph $G$ of treewidth at most $k$ satisfies $\alpha_c(G) \geq \ell_{k,c}\cdot |V(G)|$.
Similarly, for our upper bounds $u_{k,c}\geq x_{k,c}$ we shall construct an infinite set of graphs $G$ of treewidth $k$ with $\alpha_c(G) \leq \lceil u_{k,c}\cdot |V(G)|\rceil$ for each such $G$.
Our results on $x_{k,c}$ are summarized in the following theorem and illustrated in~\cref{fig:all-bounds}.

\begin{theorem}\label{thm:main}
~
    \begin{enumerate}[a]
        \item $\frac{c}{k+c+1} \, \leq \, x_{k,c} \, \leq \, \frac{c}{k+c}$ for every $c,k \geq 1$.\label{enum:general-c-k}
     
        \item $x_{1,c} = \frac{c}{1+c}$ for every $c \geq 1$.\label{enum:k1}
     
        \item $x_{k,1} = \frac{1}{k+1}$ and $x_{k,2} = \frac{2}{k+2}$ for every $k \geq 1$.\label{enum:c1-and-c2}
     
        \item $x_{2,3} = \frac{5}{9}$.\label{enum:c3-and-k2}
    \end{enumerate}
\end{theorem}

\begin{figure}
    \centering
    \includegraphics{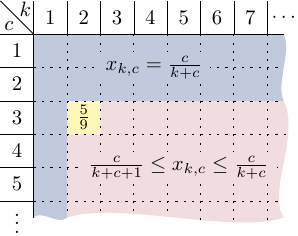}
    \caption{Illustration of the bounds on $x_{k,c} = \inf\{\alpha_c(G)/|V(G)| \colon \tw(G) = k\}$ in \cref{thm:main}.}
    \label{fig:all-bounds}
\end{figure}

That is, we show a general upper bound of $x_{k,c} \leq \frac{c}{k+c}$ and a general lower bound of $x_{k,c} \geq \frac{c}{k+c+1}$ that is just an additive $1$ in the denominator away, cf.~\Cref{thm:main}\labelcref{enum:general-c-k}.
We also show our upper bound to be tight if $k = 1$ or $c \leq 2$, cf.~\Cref{thm:main}\labelcref{enum:k1} and~\Cref{thm:main}\labelcref{enum:c1-and-c2}.
But somewhat surprisingly, in the smallest case $k=2$ and $c=3$ not covered by the previous results, we show that $x_{2,3} = \frac59$, cf.~\Cref{thm:main}\labelcref{enum:c3-and-k2}. This matches neither the general upper nor the general lower bound in~\Cref{thm:main}\labelcref{enum:general-c-k}.
We have no obvious good candidate for the true value of $x_{k,c}$ in general and wonder if there are further infinite families of parameters for which simple formulas for $x_{k,c}$ exist.

\subparagraph*{Organisation of the paper.}

After discussing some related work and previously known bounds on $x_{k,c}$, we start in \cref{sec:treewidth-intro} with a short introduction of the model for treewidth-$k$ graphs that we primarily use in our proofs.
In \cref{sec:treewidth-intro} we also explain that \cref{eq:trivial-LB} does not yield interesting lower bounds on $x_{k,c}$ if $c \neq 1$.

We then prove \cref{thm:main}\labelcref{enum:general-c-k} in \cref{sec:general-case}, \cref{thm:main}\labelcref{enum:k1} in \cref{sec:k1-case}, \cref{thm:main}\labelcref{enum:c1-and-c2} in \cref{sec:c2-case}, and \cref{thm:main}\labelcref{enum:c3-and-k2} in \cref{sec:c3-k2-case}.
Finally, we give a brief discussion of $c$-clustered independence numbers in other classes of graphs in \cref{sec:conclusions}.

\subparagraph*{Related Work.}

The $c$-clustered independence number $\alpha_c(G)$ has been considered under the name \emph{$t$-component stability number} by Broutin and Kang~\cite{BK18} in 2018 as a tool to give a lower bound on $\chi_c(G)$ (specifically for the Erd\H{o}s-R\'enyi random graph $G_{n,p}$) by using \cref{eq:trivial-LB}.
Apart from that, $\alpha_c(G)$ has (to the best of our knowledge) not been further considered or investigated.
But there is a number of equivalent or closely related concepts in the literature.
Most relevant to us are the following results of Dvo\v{r}\'ak and Wood~\cite{Wo}, and Chappell and Pelsmajer~\cite{ChPe}.

\begin{theorem}[{Dvo\v{r}\'ak and Wood~\cite[Theorem 5.3]{Wo}}]\label{wood}{\ \\}
    Let $G=(V,E)$ be a graph on $n$ vertices and treewidth at most $k$.
    If $n\leq \lfloor\frac{p}{k+1}\rfloor(c+1)+k+c-1$ and $k+1\leq p$, then there is a set $S\subseteq V$ of size $p$, such that all connected components of $G\setminus S$ have order at most $c$. 
\end{theorem}

Rearranging terms, \cref{wood} proves that
\[
    \frac{c-k}{c+1} \leq \inf\{ \frac{\alpha_c(G)}{|V(G)|} \colon \tw(G) = k\},
\]
i.e., gives the lower bound $x_{k,c} \geq \frac{c-k}{c+1}$.
As $\frac{c-k}{c+1} < \frac{c}{k+c+1}$ for all $k,c \geq 1$, the lower bound $x_{k,c} \geq \frac{c}{k+c+1}$ in~\cref{thm:main}\labelcref{enum:general-c-k} supersedes \cref{wood}.

Chappell and Pelsmajer~\cite{ChPe} investigated for $d \geq 0$ and $k \geq 1$ the largest set $S$ of vertices in any $n$-vertex graph $G$ of treewidth $k$, such that the induced subgraph $G[S]$ is a forest of maximum degree at most $d$.
Equivalently, for $d=0$, $S$ is an independent set, and for $d=1$, $S$ is a $2$-clustered set. (For $d \geq 2$ there is no equivalent correspondence in $c$-clustered sets.)

\begin{theorem}[Chappell and Pelsmajer~\cite{ChPe}]{\ \\}
    Let $G=(V,E)$ be a graph on $n$ vertices and treewidth at most $k$.
    Then there is a subset $S \subseteq V$ such that $G[S]$ has maximum degree at most $1$ and
    \[
      |S| \geq \frac{2n+2}{2k+3} \;\text{ if }k \geq 4 \qquad \text{ respectively } \qquad |S| \geq \frac{2}{k+2}n \;\text{ if }k \leq 3\text{.}
    \]
\end{theorem}

They also conjecture that their (better) bound for $k\leq 3$ should hold for all $k$.

\begin{conjecture}[{Chappell and Pelsmajer~\cite[Conjecture 13]{ChPe}}]\label{conj:Chappell-Pelsmajer}{\ \\}
    For $k \geq 0$, if $G$ is a graph of order $n$ and treewidth at most $k$, then $G$ admits an induced subgraph $G[S]$ of maximum degree at most $1$ and $|S| \geq \lceil \frac{2n}{k+2} \rceil$.
\end{conjecture}

We confirm \cref{conj:Chappell-Pelsmajer} in \cref{prop:lower-bound-c2} below.
In particular, \cref{thm:main}\labelcref{enum:c1-and-c2} indeed states that $x_{k,2} \geq \frac{2}{k+2}$.

\medskip

Let us also briefly mention some further notions that are related to the $c$-clustered sets of a graph $G=(V,E)$.
Clearly, $S \subseteq V$ is $1$-clustered (i.e., an independent set) if and only if its complement $V-S$ is a vertex cover.
Along these lines, complements of $c$-clustered sets are also known as $c$-vertex separators~\cite{Lee19}, $c$-separators~\cite{BFNNPZ23}, or $c$-component order connected sets~\cite{KL17}, and for the special case of $c=2$ as $3$-path vertex covers~\cite{BKKS11}.
Meanwhile, $2$-clustered sets appeared under the name of dissociation sets~\cite{Yan81}.

\section{Graphs of treewidth $k$ and a first observation}
\label{sec:treewidth-intro}

All graphs considered here are finite, simple, and undirected.
For a graph $G$ and a vertex $v \in V(G)$, we denote the neighborhood of $v$ in $G$ by $N_G(v) = \{u \in V(G) \colon uv \in E(G)\}$.

For our arguments it will be convenient to rely on the definition of the treewidth of a graph in terms of $k$-tree models below.
In a rooted tree $T$ with root $r$, a vertex $u$ is an \emph{ancestor} of vertex $v$ (and $v$ is a \emph{descendant} of $u$) if $u$ lies on the unique path in $T$ from $v$ to $r$.
We denote the distance between two vertices $u$ and $v$ by $\dist(u,v)$ and call vertex $u\in T$ \emph{lower} than another vertex $v \in T$ if $\dist(u,r)\geq \dist(v,r)$.
The \emph{height} of $T$ is the largest distance of any vertex in $T$ to the root plus $1$.

For $k \in \mathbb{N}$, we denote $[k] = \{1,\ldots,k\}$.

\begin{definition}
    \label{def:k-tree-model}
     A \emph{$k$-tree model} of a graph $G$ is a rooted tree $T$ with vertex set $V(T) = V(G)$, together with a labeling $L \colon V(T) \to [k+1]$ such that for every edge $uv \in E(G)$ we have
     \begin{itemize}
        \item $L(u) \neq L(v)$ and
        \item $u$ is the lowest ancestor of $v$ with label $L(u)$, or $v$ is the lowest ancestor of $u$ with label $L(v)$.
     \end{itemize}
\end{definition}

See \cref{fig:full-k-ary-tree,fig:path-on-k-and-c} for examples of graphs $G$ and some corresponding $k$-tree models $(T,L)$.
Establishing some notation, for a fixed $k$-tree model $(T,L)$ of $G$ and a vertex $v$ of $G$, the \emph{parents} of $v$ are those neighbors $u \in N_G(v)$ of $v$ in $G$ that are ancestors of $v$ in $T$.
Similarly, whenever $u$ is a parent of $v$, then $v$ is a \emph{child} of $u$.
Note that in any $k$-tree model $(T,L)$, the parents of $v$ have pairwise distinct labels and distinct from $L(v)$.
Thus, $v$ has at most $k$ parents, while $v$ may have arbitrarily many children.

While $k$-tree models are very closely related to tree decompositions, they do not appear explicitly in the literature.
A \emph{tree decomposition} of a graph $G$ is a tree $T$, together with a set $X_t \subseteq V(G)$ for each $t \in V(T)$, such that 
\begin{itemize}
    \item for every edge $uv \in E(G)$ there exists a $t \in V(T)$ with $\{u,v\} \subseteq X_t$, and
    \item for every vertex $v \in V(G)$, the set $\{t \in V(T) \colon v \in X_t\}$ induces a subtree of $T$.
\end{itemize}
The \emph{treewidth} of $G$, denoted by $\tw(G)$, is the smallest $k \in \mathbb{N}$ for which $G$ admits a tree decomposition $(T,\{X_t\}_{t \in V(T)})$ of \emph{width} $k+1$, meaning that $|X_t| \leq k+1$ for each $t \in V(T)$.

\begin{lemma}
    The treewidth $\tw(G)$ of $G$ is the smallest $k \in \mathbb{N}$ for which there exists a $k$-tree model $(T,L)$ of $G$.
\end{lemma}
\begin{proof}
    It is well-known that if $\tw(G) = k$, then $G$ admits a tree decomposition $(T,\{X_t\}_{t \in V(T)})$ of width $k+1$ with the following additional properties (called a \emph{clean tree decomposition}).
    \begin{itemize}
        \item Tree $T$ is a rooted tree, which determines for each vertex $v \in V(G)$ a \emph{root} $r(v)$ of its corresponding subtree, namely $r(v) = \mathrm{argmin}\{\dist(t,r) \colon t \in V(T), v \in X_t\}$.
        \item Every $t \in V(T)$ is the root of exactly one vertex $v \in V(G)$.
    \end{itemize}
    We fix such a clean tree decomposition $(T,\{X_t\}_{t \in V(T)})$ of width $k+1$.
    
    Another classic fact is that $G$ admits a proper vertex coloring $\phi\colon V(G) \to [k+1]$ with $k+1$ colors, such that whenever $\phi(u) = \phi(v)$ for distinct vertices $u,v \in V(G)$, then there is \emph{no} $t \in V(T)$ with $\{u,v\} \subseteq X_t$.
    Now, we define the labeling $L \colon V(T) \to [k+1]$ as $L(t) = \phi(r^{-1}(t))$, where $r^{-1}(t)$ is the vertex $v \in V(G)$ with $r(v) = t$.
    Renaming each $t$ to $r^{-1}(t)$ gives $V(T) = V(G)$, and together with the labeling $L$, this is a $k$-tree model of $G$.
    
    \medskip
    
    Conversely, assume that $(T,L)$ is a $k$-tree model of $G$.
    Then we obtain a (clean) tree decomposition $(T,\{X_t\}_{t \in V(T)})$ of width $k+1$ of $G$ by setting $X_t = \{t\} \cup \{v \in V(G) \colon v \text{ is a parent of } t\}$ for each $t \in V(T)$.
\end{proof}

As mentioned above, it is convenient for us to phrase our constructions and proofs in terms of $k$-tree models.
As the $c$-clustered independence number $\alpha_c(G)$ is monotone under taking subgraphs, we may always assume that $G$ is \emph{edge-maximal} with respect to its given $k$-tree model $(T,L)$.
That is, two vertices $u,v \in V(G)$ form an edge in $G$ if and only if $u$ is a parent of $v$ or $v$ is a parent of $u$.
We remark that, while $G$ might have several $k$-tree models (and might even be edge-maximal with respect to only some of these), every $k$-tree model $(T,L)$ has exactly one corresponding edge-maximal graph $G$.

\smallskip

Let us come back to the $c$-clustered chromatic number $\chi_c(G)$, the $c$-clustered independence number $\alpha_c(G)$, and \cref{eq:trivial-LB} saying that $\alpha_c(G)/|V(G)| \geq 1 / \chi_c(G)$ for every graph $G$.
If $(T,L)$ is a $k$-tree model of $G$, then in particular $L$ is a proper vertex coloring of $G$ with $k+1$ colors.
Hence, $\chi_1(G) = \chi(G) \leq \tw(G)+1$ for every graph $G$, which implies with \cref{eq:trivial-LB} that
\[
    \frac{\alpha_1(G)}{|V(G)|} \geq \frac{1}{k+1} \text{ if } \tw(G) = k \quad \text{ and thus } \quad x_{k,1} \geq \frac{1}{k+1}.
\]
It is easy to see that in fact $x_{k,1} = \frac{1}{k+1}$, cf.~\cref{thm:main}\labelcref{enum:c1-and-c2}.
However, \cref{eq:trivial-LB} does not give anything better for $c > 1$, due to the following.

\begin{observation}\label{obs:c-clustered-coloring}
    For any $k,c$, let $T$ be the full $c$-ary tree with root $r$ and height $k+1$, and $L \colon V(T) \to [k+1]$ the labeling given by $L(v) = \dist(v,r) + 1$.
    See \cref{fig:full-k-ary-tree} for an example.
    Then the edge-maximal graph $G$ with $k$-tree model $(T,L)$ satisfies $\chi_c(G) = k+1$.

    \begin{figure}
        \centering
        \includegraphics{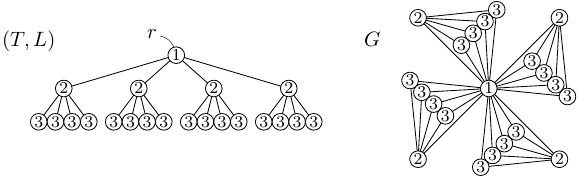}
        \caption{
            A graph $G$ (right) with a $k$-tree model $(T,L)$ (left) with $\chi_c(G) = k+1$ for $c = 4$, $k = 2$.
        }
        \label{fig:full-k-ary-tree}
    \end{figure}

    In fact, if $\phi$ is any $c$-clustered coloring of $G$ and the root $r$ receives color $i$, then at least one of the $c$ subtrees below $r$ contains no vertex of color $i$, and it follows by induction on $k$ that there are at least $k+1$ distinct colors.
\end{observation}

\cref{obs:c-clustered-coloring} shows that for $c \neq 1$ we do not get anything better than the obvious lower bound $x_{k,c} \geq x_{k,1} \geq \frac{1}{k+1}$ from considering $c$-clustered chromatic numbers and the simple averaging argument in~\cref{eq:trivial-LB}.

\section{General bounds for all $k$ and $c$}
\label{sec:general-case}

Here we consider the case of any integers $k,c \geq 1$, i.e., we prove \cref{thm:main}\labelcref{enum:general-c-k}, starting with the lower bound.

\begin{proposition}\label{prop:lower-bound}
    For every $k,c \geq 1$ and every graph $G$ of treewidth $k$ we have $\alpha_c(G) \geq \frac{c}{k+c+1}|V(G)|$, i.e., $x_{k,c} \geq \frac{c}{k+c+1}$.
\end{proposition}
\begin{proof}
    Fix $G = (V,E)$ to be any graph of treewidth $k$.
     We proceed by induction on $n = |V|$ and find a $c$-clustered set $S$ in $G$ of size $|S| \geq \frac{c}{k+c+1}n$.
     For the induction base, the case $n \leq c$, it is enough to take $S = V$.
     So assume that $n > c$.
     Let $(T,L)$ be a $k$-tree model of $G$, and assume without loss of generality that $G$ is edge-maximal with respect to $(T,L)$.
     Let $v$ be a lowest vertex in $T$ that has at least $c$ descendants.
     Let $A$ be the set of descendants of $v$ and $B = N_G(v) - A$.
     Then $|A| \geq c$, $A$ is a $c$-clustered set in $G$ (by the minimality in the choice of $v$), and $|B| \leq k$.
     
     Further, we claim that no vertex in $A$ is adjacent to any vertex in $G' = G- (A \cup B \cup \{v\})$.
     In fact, if $u \in A$ is adjacent to $w \notin A \cup \{v\}$, then $L(u) \neq L(w)$ and $w$ is an ancestor of $v$ in $T$.
     If $L(w) \neq L(v)$, then also $v$ is adjacent to $w$ by the edge-maximality of $G$, i.e., $w \in B$, as desired.
     On the other hand, $L(w) = L(v)$ would contradict the fact that $w$ is the lowest ancestor of $u$ with that label.
    
     By induction on $G'$, there is a $c$-clustered set $S'$ of at least $\frac{c}{k+c+1}(n - |A| - |B|-1)$ vertices in $G'$.
     Then $S = S' \cup A$ is a $c$-clustered set of size at least:
     \begin{center}    
         \begin{tabular}{cl}
             ~&$\frac{c}{k+c+1}(n - |A| - |B|-1) + |A|$ \\
             $=$&$\frac{c}{k+c+1}n + \frac{k+1}{k+c+1}|A| - \frac{c}{k+c+1}(|B|+1)$ \\
             $\geq$&$\frac{c}{k+c+1}n + \frac{(k+1)c}{k+c+1} - \frac{c(k+1)}{k+c+1}$ \\
             $=$& $\frac{c}{k+c+1}n.$
         \end{tabular}
     \end{center}
\end{proof}

The upper bound on $x_{k,c}$ is a simple construction.

\begin{observation}\label{obs:upper-bound}
    For any $k,c$, let $T = [v_1,\ldots,v_{k+c}]$ be a path on $k+c$ vertices rooted at $v_1$, and $L\colon V(T) \to [k+1]$ the labeling given by $L(v_i) = i$ for $i = 1,\ldots,k$ and $L(v_i) = k+1$ for $i = k+1,\ldots,k+c$.
    See \cref{fig:path-on-k-and-c} for an example.
    Then the edge-maximal graph $G$ with $k$-tree model $(T,L)$ satisfies $\alpha_c(G) = c$ and $|V(G)| = k+c$.

    \begin{figure}
        \centering
        \includegraphics{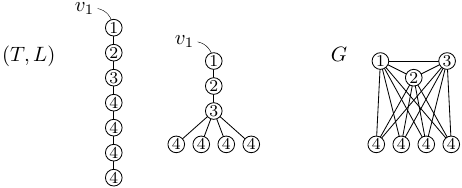}
        \caption{A graph $G$ (left) with two different $k$-tree models $(T,L)$ (left) with $\alpha_c(G) = c$ and $|V(G)| = k+c$, i.e., $\alpha_c(G) = \frac{c}{k+c}|V(G)|$, for $k=3$, $c=4$.}
        \label{fig:path-on-k-and-c}
    \end{figure}

    In fact, $v_1,\ldots,v_k$ are universal vertices in $G$ and hence any set of $c+1$ vertices in $G$ induces a connected subgraph of size $c+1$, i.e., is not $c$-clustered.
\end{observation}

Taking arbitrary vertex-disjoint unions of the graph $G$ in \cref{obs:upper-bound}, it follows that $x_{k,c} \leq \frac{c}{k+c}$.
So \cref{prop:lower-bound} and \cref{obs:upper-bound} together prove \cref{thm:main}\labelcref{enum:general-c-k}.

\section{The case of graphs of treewidth $1$}
\label{sec:k1-case}

We shall prove \cref{thm:main}\labelcref{enum:k1}.
For this we shall argue that, for the case of treewidth $k=1$, we can improve the lower bound in \cref{prop:lower-bound} by just slightly changing its proof.
To this end, we remark that any graph $G$ with at least one edge admits a $k$-tree model $(T,L)$ with $k = \tw(G)$ in which the two endpoints of any edge in $T$ have distinct labels in $L$.
In fact, if $v$ is the immediate ancestor of $u$ in $T$ and $L(u) = L(v)$, we can change the tree (keeping all labels) by hanging the subtree $T_u$ rooted at $u$ under the lowest ancestor $w$ of $u$ with $L(w) \neq L(u)$.
If there is no such ancestor, then the vertices in $T_u$ are not connected to the remaining graph, and we can permute the labeling in $T_u$ to give $u$ a label distinct from $L(v)$.
In either case, the result is still a $k$-tree model of $G$ with fewer problematic edges.

\begin{proposition}\label{prop:lower-bound-k1}
    For every $c \geq 1$ and every graph $G$ of treewidth $1$ we have $\alpha_c(G) \geq \frac{c}{1+c}|V(G)|$, i.e., $x_{1,c} \geq \frac{c}{1+c}$.
\end{proposition}
\begin{proof}
     We proceed by induction on $n = |V|$ and find a $c$-clustered set $S$ in $G$ of size $|S| \geq \frac{c}{1+c}n$.
     For the induction base, the case $n \leq c$, it is enough to take $S = V$.
     So assume that $n > c$.
     Let $(T,L)$ be a $1$-tree model of $G$ with the property that any edge of $T$ connects two vertices of distinct label in $L$.
     As in the proof of \cref{prop:lower-bound}, let $v$ be a lowest vertex in $T$ with a set $A$ of at least $c$ descendants.
     Then $A$ is a $c$-clustered set of size $|A| \geq c$, and no vertex in $A$ is adjacent to any vertex in $G' = G - (A \cup \{v\})$.
     In fact take any $w \in A$ and assume by symmetry that $L(v) = 1$.
     If $L(w) = 2$, then $v$ is the only ancestor of $w$ in $N_G(w)$.
     And if $L(w) = 1$, the ancestor of $w$ in $N_G(w)$ has label $2$ and is the immediate ancestor of $w$ in $T$ by our additional assumption on the $1$-tree model.
    
     Now by induction on $G'$, there is a $c$-clustered set $S'$ of at least $\frac{c}{1+c}(n-|A|-1)$ vertices in $G'$.
     Then $S = S' \cup A$ is a $c$-clustered set of size at least:
     \begin{center}
         \begin{tabular}{cl}
              ~ & $\frac{c}{1+c}(n-|A|-1) + |A|$ \\
              $=$ & $\frac{c}{1+c}n + \frac{1}{1+c}|A| - \frac{c}{1+c}$\\
              $\geq$ & $\frac{c}{1+c}n + \frac{c}{1+c} - \frac{c}{1+c}$\\
              $=$ & $\frac{c}{1+c}n.$
         \end{tabular}
     \end{center}
\end{proof}

\section{The case of $1$-clustered and $2$-clustered sets}
\label{sec:c2-case}

We shall prove \cref{thm:main}\labelcref{enum:c1-and-c2}.
In fact, we already have the upper bound $x_{k,c} \leq \frac{c}{k+c}$ and need to prove a matching lower bound when $c \leq 2$.
For $c=1$, this is already given by~\cref{eq:trivial-LB} and it remains to consider the case $c=2$ here.

\begin{proposition}\label{prop:lower-bound-c2}
    For every $k \geq 1$ and every graph $G$ of treewidth $k$ we have $\alpha_2(G) \geq \frac{2}{k+2}|V(G)|$, i.e., $x_{k,2} \geq \frac{2}{k+2}$.
\end{proposition}
\begin{proof}
     Let $G = (V,E)$ be any graph of treewidth $k$.
     We proceed by induction on $n = |V|$ and find a $2$-clustered set $S$ in $G$ of size $|S| \geq \frac{2}{k+2}n$.
     For the induction base, the case $n \leq k+2$, it is enough to let $S$ be any pair of vertices.
     So assume that $n \geq k+3$ and let $(T,L)$ be a $k$-tree model of $G$.
     Without loss of generality assume that $G$ is edge-maximal with model $(T,L)$.
     In particular, the parents of each vertex form a clique. 
    
     We describe a procedure that gradually determines which vertices of $G$ to \emph{take}, i.e., include them in the desired set $S$, and which to \emph{discard}, i.e., include them in another set $D$.
     All such decisions will be irrevocable, and eventually $S$ and $D$ will partition $V$.
     During the course of the procedure, vertices in $V - (S \cup D)$ are called \emph{undecided}.
     In order to ensure that $S$ is a $2$-clustered set we maintain the following invariant for every vertex~$v$.
    
     \begin{enumerate}[{I}1]
         \item If $v \in S$, then $v$ has no undecided children and at most one child in $S$. If $v$ has a child in $S$, then all parents of $v$ are in $D$. \label{inv:2-clustered}
     \end{enumerate}
    
     Note that this indeed ensures $S$ to be a $2$-clustered set.
     In order to bound the size $|S|$ of $S$ in terms of $n$ the number of vertices in $G$, we use \emph{tokens} placed on vertices. Initially, there are no tokens.
     Taking an undecided vertex $v$ grants $k$ tokens, which we can distribute on the remaining undecided vertices $V-(S \cup D \cup \{v\})$.
     Discarding an undecided vertex $v$ costs $2$ tokens, which we can remove from $v$ or other undecided vertices.
     Thus, as soon as there are no more undecided vertices, i.e., $V = S \cup D$, we can conclude that
     \[
        k|S| \geq \#\text{tokens granted by taking vertices} \geq \#\text{tokens spend by discarding vertices} \geq 2|D|
     \]
     and thus
     \[
        \frac{k}{2} |S| \geq |D| = n-|S| \quad \Rightarrow \quad \frac{k+2}{2}|S| \geq n \quad \Rightarrow \quad |S| \geq \frac{2}{k+2}n,
     \]
     as desired.
     
     At intermediate states we allow a negative token count at undecided vertices.
     But still, in order to discard a vertex $v$, we first must redistribute tokens so that the token count at $v$ is at least $2$, and in order to take vertex $v$ the token count at $v$ must be at least $0$.
    
     Initially $S = D = \emptyset$, i.e., all vertices of $G$ are undecided.
     Throughout we maintain a $k$-tree model $(T,L)$ for $G' = G - D$, i.e., the induced subgraph of $G$ on all taken and undecided vertices.
     By discarding a vertex $v$, we remove $v$ from the current $k$-tree model of $G'$ by contracting $v$ into its immediate ancestor in $T$, keeping the labels at all vertices (except the removed $v$).
     Note that this indeed results in a $k$-tree model of $G' - v$ with the set of parents of each vertex forming a clique.
     For convenience we shall denote the new $k$-tree model again by $(T,L)$.
     Recall that a parent of $v$ is a vertex $u \in N_G(v)$ that is an ancestor of $v$ in $T$.
     Each vertex has at most $k$ parents but by discarding vertices, we may reduce the number of parents of other undecided (or taken) vertices.
    
     For any undecided vertex $v$, let $t_v$ denote the number of tokens at $v$, $p_v$ the number of parents of $v$, and $s_v$ the number of children of $v$ that are in $S$.
     We maintain the following invariants for every undecided vertex $v$:
    
     \begin{enumerate}[{I}1]
        \setcounter{enumi}{1}
        \item If $v$ has undecided children or $s_v \geq 2$, then $t_v \geq s_v$.\label{inv:inner-vertex}
        
        \item If $v$ has no undecided children, then $t_v \geq s_v + p_v - k$. \label{inv:leaf-vertex}
     \end{enumerate}
    
     Note that these invariants hold initially when $S = D = \emptyset$.
     
     Let $v$ be a lowest undecided vertex in $T$.
     I.e., all children of $v$ (if any) are in $S$.
     (Recall that vertices in $D$ are removed from the graph and the $k$-tree model.)
    
     \begin{enumerate}[{Case} 1]
        \item No children of $v$ are in $S$.
    
        We have $s_v = 0$ and by \ref{inv:leaf-vertex} there are $t_v \geq s_v + p_v - k = p_v - k$ tokens at $v$.
        We take $v$, i.e., add $v$ to the set $S$, which grants $k$ tokens, pay $k-p_v$ of these tokens to have the token count at $v$ at $0$, and spend the remaining $p_v$ tokens by putting $1$ token on each of the $p_v$ parents of $v$.
        This maintains the invariants.
    
        \item At least two children of $v$ are in $S$.
    
        We have $s_v \geq 2$ and by \ref{inv:inner-vertex} there are $t_v \geq s_v \geq 2$ tokens at $v$.
        We spend $2$ tokens from $v$ to discard $v$, i.e., add $v$ to the set $D$.
        The invariants are again maintained.
     \end{enumerate}
    
     For the remainder we may assume that each undecided vertex either has an undecided child or exactly one child in $S$.
     Let $v$ be again a lowest undecided vertex in $T$, i.e., we have $s_v = 1$, and let $w$ denote the lowest parent of $v$ in $T$.
    
     \begin{enumerate}[{Case} 1]
        \setcounter{enumi}{2}
        \item $w$ has $t_w \geq 1$ tokens.
    
        We take one token from $v$ and one token from $w$ and use these $2$ tokens to discard $w$.
        This maintains \ref{inv:leaf-vertex} as $v$ loses one token but also one parent.
        
        \item $w$ has another undecided child $u$ different from $v$.
    
        We have that $u$ is also a lowest undecided vertex in $T$, since $v$ and $u$ have the same parent $w$ and $v$ is a lowest undecided vertex.
        In particular, $u$ has no undecided children and we care about \ref{inv:leaf-vertex} at $u$.
        By \ref{inv:inner-vertex} $w$ has $t_w \geq s_w \geq 0$ tokens.
        We take one token from $u$ and one from $v$ and use these $2$ tokens to discard $w$.
        This maintains \ref{inv:leaf-vertex} as $u$ and $v$ each lose one token but also one parent.
     \end{enumerate}
    
     For the remainder we may assume that $v$ is the only undecided child of $w$ and that $w$ has $t_w \leq 0$ tokens.
     By \ref{inv:inner-vertex} we have $t_w \geq s_w \geq 0$, i.e., $t_w = s_w = 0$ and $w$ has no children in $S$.
     Thus we are left with the following case.
    
     \begin{enumerate}[{Case} 1]
        \setcounter{enumi}{4}
        \item $w$ has only $v$ as undecided child and no children in $S$.
    
        In this case we rely on induction.
        We perform a local modification on the $k$-tree model $(T,L)$, replacing $v$ and $w$ by a single new vertex $u$.
        To this end, let $P$ denote the set of all parents of $v$.
        Contract $v$ and $w$ into the immediate ancestor $z$ of $w$ in $T$.
        Add a new vertex $u$ as a leaf to $z$, make $P - \{w\}$ the parents of $u$ by giving $u$ the label of $v$, put $u$ into $S$ and put $1$ additional token on each vertex in $P - \{w\}$.
        This modification costs us $|P|-1 = p_v - 1$ tokens and provides us with $t_v+t_w$ tokens from $v$ and $w$, causing the total cost
        \[
         t_v + t_w - (p_v - 1) \overset{\ref{inv:leaf-vertex}}{\geq} s_v + p_v - k + 0 - p_v + 1 = 2 - k,
        \]    
        which we will balance out by taking one out of $v,w$ and discarding the other. 
        Our choice will be determined by induction. 
        For now, observe that the new situation satisfies our invariants and has one vertex less.
        By induction we get a partition $(S',D')$ of the set $V' = V - (D \cup \{v,w\})$ of all remaining vertices such that $S \subseteq S'$ and $S'$ is $2$-clustered in the modified graph.
        Recall that we have put the artificial vertex $u$ into $S$, and hence $u \in S'$.
        We want to replace $u \in S'$ by one of $v,w$ and discard the other so that the result is $2$-clustered in the original graph.
        
        If $S' \cap P = \emptyset$, then we discard $w$.
        Now, all parents of $v$ are in $D \cup D'$.
        Thus, we can replace $u \in S'$ by $v$ which has only $s_v = 1$ child in $S'$.
        Taking $v$ grants us $k$ tokens and discarding $w$ costs us $2$ tokens, balancing the $2-k$ deficit from the modification.
    
        If $S' \cap P \neq \emptyset$, then let $a$ be a vertex in $S' \cap P$.
        We want to discard $v$ and replace $u \in S'$ by $w$.
        To see that this is possible, let $A$ denote the set of all parents of $w$.
        Then $P - \{w\} \subseteq A$.
        Recall that $P-\{w\}$ is the set of parents of the artificial vertex $u$.
        As $S'$ is $2$-clustered, $u,a \in S'$, and $A$ forms a clique in the original graph $G$, it follows that $G[S' - \{u\}]$ has a component of size $1$ only consisting of vertex $a$.
        As no child of $w$ is in $S' - \{u\}$, we can indeed take $w$ and discard $v$.
        Again, this grants us $k$ tokens and costs us $2$ tokens, balancing the $2-k$ deficit from the modification.
     \end{enumerate}
    
     Observe that by our invariants, this complete case distinction concludes the proof.
\end{proof}

\section{The case of $3$-clustered sets in graphs of treewidth~$2$}
\label{sec:c3-k2-case}

We shall show in this section that for $k=2$ and $c=3$, the smallest $c$-clustered independence number of $n$-vertex graphs of treewidth $k$ is $\lceil \frac{5}{9}n \rceil$, i.e., $x_{2,3} = \frac59$.
Note that for $k=2$ and $c=3$ we have
\[
    \frac{c}{k+c+1} = \frac{1}{2} < \frac{5}{9} < \frac{3}{5} = \frac{c}{k+c},
\]
i.e., this value lies strictly between the general lower and upper bound in \cref{thm:main}\labelcref{enum:general-c-k}.

We start with an explicit construction for the upper bound.
Let $G_1$ be the $10$-vertex graph shown in the left of \cref{fig:k2_c_3_counterexample}.
For an integer $i \geq 2$ let $G_i$ be the graph obtained from $i$ copies of $G_1$ by identifying vertex $v_6$ of each copy (except the last) with the vertex $v_1$ of the previous copy, and adding an edge between vertex $v_5$ of each copy (except the last) with vertex $v_2$ of the previous copy.
See the right of \cref{fig:k2_c_3_counterexample} for an illustration.
Note that $G_i$ has $9i+1$ vertices and treewidth $2$.
In fact, each $G_i$ is outerplanar as shown in \cref{fig:k2_c_3_counterexample}.

\begin{figure}
    \centering
    \includegraphics{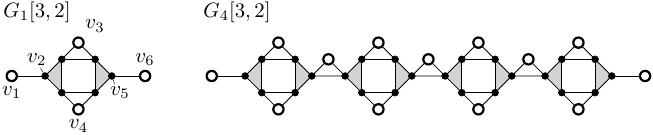}
    \caption{Illustration of the graph $G_i$, $i \geq 1$, with $9i+1$ vertices and no $3$-clustered set of size more than $5i+2$.}
    \label{fig:k2_c_3_counterexample}
\end{figure}

\begin{proposition}\label{prop:x23-upper-bound}
    For each $i \geq 1$ we have $\alpha_3(G_i) \leq 5i+1 = \lceil \frac{5}{9}|V(G_i)| \rceil$.
\end{proposition}
\begin{proof}
     Let $S \subseteq V$ be a maximum $3$-clustered set in $G_i$.
     Let $A$ be the set of all vertices $v_1,v_3,v_4,v_6$ from all copies of $G_1$.
     This is, $A$ is the set of $3i+1$ vertices shown in white in \cref{fig:k2_c_3_counterexample}.
     First, we claim that without loss of generality we may assume that $A \subseteq S$.
     So assume that $v \in A$ with $v \notin S$.
     By maximality of $S$, there is a neighbor $u \in N_G(v)$ with $u \in S$.
     Note that $u \notin A$, as $A$ is an independent set in $G_i$.
     The vertices in $A$ are the simplicial vertices of $G_i$, giving that $N_G(v) - u \subseteq N_G(u) - v$.
     Hence $S' = S - u \cup \{v\}$ is also a $3$-clustered set of the same size with $|A - S'| < |A - S|$.
     
     Second, the vertices in $V - A$ are partitioned into $2i$ vertex-disjoint triangles, highlighted in gray in \cref{fig:k2_c_3_counterexample}.
     Given that $A \subseteq S$ and $S$ is $3$-clustered, observe that from each such triangle there is at most one vertex in $S$.
     Thus $|S| \leq |A| + 2i = 5i+1$, as desired.
\end{proof}

By \cref{prop:x23-upper-bound} we have $x_{2,3} \leq \frac59$.
For the proof of the lower bound $x_{2,3} \geq \frac{5}{9}$ we shall show that every graph $G$ of treewidth~$2$ admits a $3$-clustered set of size at least $\frac59 |V(G)|$.
We present this proof without using $k$-tree models.
Instead, it will be more convenient to work with a \emph{$2$-tree $G$ rooted at some edge $e_0 = u_0v_0$}, i.e., a graph that can be constructed starting with $e_0$ by iteratively adding a new vertex to the endpoints of an already constructed edge. We consider \emph{cut pairs} in $G$, i.e., edges $e = uv$ such that $G-u-v$ is disconnected.
There is a connected component of $G-u-v$ for each $w \in N(u) \cap N(v)$ where those in a different component than $u_0,v_0$ are called the children of $u$ and $v$.
In terms of a $2$-tree model $(T,L)$ of $G$ rooted at $u_0$, the children of $e=uv$, say with $v$ below $u$ in $T$, are the highest vertices $w$ in the subtree below $v$ with the label $L(w) \notin \{L(u),L(v)\}$.

\begin{proposition}
    For every graph $G=(V,E)$ of treewidth at most $2$, we have $\alpha_3(G) \geq \frac{5}{9}|V|$.
\end{proposition}
\begin{proof}
     Since every graph of treewidth $2$ is subgraph of a $2$-tree, we can assume without loss of generality that $G$ is a $2$-tree.
     We proceed by induction on $n = |V|$.
     If $n = 3$, then clearly $\alpha_3(G) = 3 \geq \frac{5}{9}\cdot 3$.
     So assume for the remainder that $n \geq 4$.
    
     We root $G$ at an arbitrary edge $e_0 = u_0v_0$, and for any edge $e = uv$ in $G$ call the common neighbors of $u$ and $v$ that are not in the component of $G - u - v$ that contains $u_0$ or $v_0$ the \emph{children} of $e$.
     Every vertex $w \notin \{u_0,v_0\}$ is the child of exactly one edge $e = uv$ and we call $u$ and $v$ the \emph{parents} of $w$.
     Note that for every edge one of its endpoints is a parent of the other endpoint (where by convention $u_0$ is the parent of $v_0$).
     Throughout this proof for each edge we always list the parent first, i.e., for any edge $uv$ vertex $u$ is a parent of vertex $v$.
    
     For each edge $e = uv$ and each subset $W$ of children of $e$, let $G_W$ denote the subgraph of $G$ induced by $W \cup \{u,v\}$ and all vertices that have an ancestor in $W$.
     Note that $G_W$ is a $2$-tree and we consider it rooted at edge $uv$.
     See \cref{fig:2-tree-illustration} for a schematic illustration.
    
     \begin{figure}
         \centering
         \includegraphics{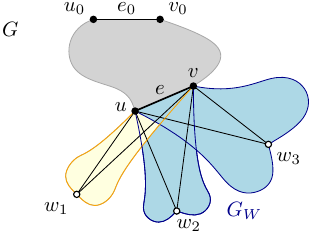}
         \caption{A $2$-tree $G$ rooted at edge $e_0 = u_0v_0$, an edge $e = uv$ in $G$, its children $w_1,w_2,w_3$, and the $2$-tree $G_W$ for $W = \{w_2,w_3\}$ (blue).}
         \label{fig:2-tree-illustration}
     \end{figure}
    
     Our goal is to find an edge $e = uv$ and a subset $W$ of children of $e$ for which $G_W$ admits a $3$-clustered set $S_W$ with $S_W \cap \{u,v\} = \emptyset$ and $|S_W| \geq \frac{5}{9}|V(G_W)|$.
     Let us call such a set $W$ a \emph{good} set.
     For simplicity, if $|W|=1$, i.e., $W = \{w\}$, we also write $S_w$ and $G_w$ instead of $S_{\{w\}}$ and $G_{\{w\}}$.
     Once we have a good set $W$, the result easily follows by induction on $G' = G - V(G_W)$.
     In fact, let $S'$ be a $3$-clustered set in $G'$ of size $|S'| \geq \frac59 |V(G')|$, then $S = S' \cup S_W$ is a $3$-clustered set in $G$ of size
     \[
        |S| = |S'| + |S_W| \geq \frac59 |V(G')| + \frac59 |V(G_W)| = \frac59 |V(G)|.
     \]
     To find a good set, we consider edges in bottom-up order, that is, descendants before ancestors, starting with the edges that have no children.
     For each considered edge $e = uv$ with $W$ being a subset of all children of $e$, we store a $3$-clustered set $S_W$ of $G_W$ with $S_W \cap \{u,v\} = \emptyset$.
     We shall ensure that $|S_W| \geq \frac{5}{9}|V(G_W - \{u,v\})| = \frac{5}{9}(|V(G_W)| - 2)$.
     With this in mind, we define
     \[
        \mathrm{sp}(W) = 9|S_W| - 5(|V(G_W)| - 2)
     \]
     as the \emph{surplus} of $W$.
     Thus we shall ensure that $\mathrm{sp}(W) \geq 0$, while if $\mathrm{sp}(W) \geq 10$, then
     \[
        9|S_W| \geq 5(|V(G_W)|-2) + 10 = 5|V(G_W)| \quad \text{and thus} \quad |S_W| \geq \frac59 |V(G_W)|.
     \]
     In other words, if $\mathrm{sp}(W) \geq 10$, then $W$ is good.
    
     Looking for good sets, let us assume again that we currently consider edge $e=uv$ with a subset $W$ of children of $e$.
     Besides determining $S_W$ and thereby $\mathrm{sp}(W)$, we also define the \emph{threat} at $u$ from $W$, denoted by $\mathrm{th}_W(u)$ as the total size of all components of $G[S_W]$ that contain a neighbor of $u$, i.e.,
     \[
        \mathrm{th}_W(u) = \# \{ x \in S_W \colon x \in C, C \text{ a component of } G[S_W], N(u) \cap C \neq \emptyset\}.
     \] 
     In other words, $S_W \cup \{u\}$ is a $3$-clustered set if and only if $\mathrm{th}_W(u) \leq 2$, i.e., the threat at $u$ from $W$ is at most $2$.
     The threat $\mathrm{th}_W(v)$ at the other endpoint $v$ of $e$ is defined analogously.
     For brevity, let us combine surplus and threats and simply say that
     \[
        \text{the \emph{type} of $W$ is $_\alpha(s)_\beta$ with $\alpha = \mathrm{th}_W(u)$, $s = \mathrm{sp}(W)$, $\beta = \mathrm{th}_W(v)$.}
     \] 
     For example, if $W = \{w\}$ and $w$ has no children, then we set $S_W = \{w\}$ and thus $\mathrm{sp}(W) = 4$, $\mathrm{th}_W(u) = \mathrm{th}_W(v) = 1$, which gives $W$ the type $_1(4)_1$.
     Note that $_\alpha(s)_\beta \neq {}_\beta(s)_\alpha$ for $\alpha \neq \beta$ because we assume that $u$ is a parent of $v$ and not vice versa.
    
     With these definitions in place, a good set $W$ is one whose type $_\alpha(s)_\beta$ satisfies $s \geq 10$.
     In the entire argument below, we shall compute the types of sets in a bottom-up approach and either find a good set or encounter and store one of the following nine different types:
     \begin{equation}
        _0(0)_0 \quad _1(4)_1, \quad _2(8)_2, \quad _1(7)_2, \quad _2(7)_1, \quad _1(6)_1, \quad _1(9)_2, \quad _2(9)_1, \quad _1(8)_1\label{eq:all-types}
     \end{equation}
     The first type $_0(0)_0$ is used only for $W = \emptyset$ with the corresponding $3$-clustered set $S_W = \emptyset$ which has a surplus of $0$ and a threat of $0$ at each parent.
     To start the process, we set this type at every edge $e$ of $G$ that has \emph{no} children, i.e., for the set $W = \emptyset$ of all children of $e$.
    
     We proceed to consider a single vertex $w$ with parent edge $uv$, and assume that we already determined the type $_{\alpha_1}(s_1)_{\beta_1}$ of the set $X$ of all children of $uw$ and the type $_{\alpha_2}(s_2)_{\beta_2}$ of the set $Y$ of all children of $vw$, and that both these types are among the nine types in \eqref{eq:all-types}.
     In each of the cases below, we shall find a good set or define a $3$-clustered set $S_w$ such that the type of $W = \{w\}$ is again one of the nine types in \eqref{eq:all-types}.
     Note that the total threat at $w$ from the types of $X$ and $Y$ is $\beta_1 + \beta_2$.
    
     \begin{enumerate}[{Case} 1]
        \item $\beta_1 + \beta_2 \leq 2$.{\ \\}
            We set $S_w = S_X \cup S_Y \cup \{w\}$.
            Since $\beta_1 + \beta_2 \leq 2$, the set $S_w$ is indeed a $3$-clustered set.
            The corresponding surplus is calculated by
            \begin{multline*}
                \mathrm{sp}(\{w\}) = 9|S_w|-5(|V(G_w)|-2) = 9(|S_X|+|S_Y|+1) - 5(|V(G_X)|-2+|V(G_Y)|-2+1)\\
                = 9|S_X|-5(|V(G_X)|-2) + 9|S_Y| - 5(|V(G_Y)|-2) + 9 - 5 = s_1+s_2 + 4,
            \end{multline*}
             i.e., the sum of the two surpluses plus $4$.
            Observe from the list \eqref{eq:all-types} of all types, that $\mathrm{sp}(\{w\}) \geq 10$ and thus $\{w\}$ a good set, unless we are in one of the following cases.
            \begin{itemize}
                \smallskip
                
                \item If $_{\alpha_1}(s_1)_{\beta_1} = {}_0(0)_0$ and $_{\alpha_2}(s_2)_{\beta_2} = {}_0(0)_0$, then the type of $\{w\}$ is $_1(4)_1$.
                
                \item If $_{\alpha_1}(s_1)_{\beta_1} = {}_0(0)_0$ and $_{\alpha_2} (s_2)_{\beta_2} = {}_1(4)_1$, then the type of $\{w\}$ is $_2(8)_2$.
    
                \item Symmetrically, if $_{\alpha_1}(s_1)_{\beta_1} = {}_1(4)_1$ and $_{\alpha_2}(s_2)_{\beta_2} = {}_0(0)_0$, the type of $\{w\}$ is $_2(8)_2$.
    
                \smallskip
            \end{itemize}
            In each case the type of $\{w\}$ is again on the list \eqref{eq:all-types}.
    
        \item $\beta_1 + \beta_2 \geq 3$.{\ \\}
            In this case we set $S_w = S_X \cup S_Y$ and calculate the surplus 
            \begin{multline*}
                \mathrm{sp}(\{w\}) = 9|S_w|-5(|V(G_w)|-2) = 9(|S_X|+|S_Y|)-5(|V(G_X)|-2+|V(G_Y)|-2+1)\\
                = 9|S_X|-5(|V(G_X)|-2) + 9|S_Y| - 5(|V(G_Y)|-2) - 5 = s_1+s_2 - 5,
            \end{multline*}
            i.e., the sum of the two surpluses minus $5$.
            For the threats at $u$ and $v$, we have $\mathrm{th}_w(u) = \alpha_1$ and $\mathrm{th}_w(v) = \alpha_2$.
            Again, observe from the list \eqref{eq:all-types} of all types, that $\mathrm{sp}(\{w\}) \geq 10$ and thus $\{w\}$ is a good set, unless we are in one of the following cases.
            \begin{itemize}
                \smallskip
                
                \item If $_{\alpha_1}(s_1)_{\beta_1} = {}_1(4)_1$ and $_{\alpha_2}(s_2)_{\beta_2} = {}_1(7)_2$, then the type of $\{w\}$ is $_1(6)_1$.
    
                \item If $_{\alpha_1}(s_1)_{\beta_1} = {}_1(4)_1$ and $_{\alpha_2}(s_2)_{\beta_2} = {}_1(9)_2$, then the type of $\{w\}$ is $_1(8)_1$.
    
                \item If $_{\alpha_1}(s_1)_{\beta_1} = {}_1(4)_1$ and $_{\alpha_2}(s_2)_{\beta_2} = {}_2(8)_2$, then the type of $\{w\}$ is $_1(7)_2$.
    
                \smallskip
    
                \item If $_{\alpha_1}(s_1)_{\beta_1} = {}_1(6)_1$ and $_{\alpha_2}(s_2)_{\beta_2} = {}_1(7)_2$, then the type of $\{w\}$ is $_1(8)_1$.
                
                \item If $_{\alpha_1}(s_1)_{\beta_1} = {}_1(6)_1$ and $_{\alpha_2}(s_2)_{\beta_2} = {}_2(8)_2$, then the type of $\{w\}$ is $_1(9)_2$.
    
                \smallskip
            \end{itemize}
            Here we omitted the symmetric cases, such as if $_{\alpha_1}(s_1)_{\beta_1} = {}_2(8)_2$ and $_{\alpha_2}(s_2)_{\beta_2} = {}_1(4)_1$, then the type of $\{w\}$ is $_2(7)_1$.
            As before, in each case the type of $\{w\}$ is again on the list \eqref{eq:all-types}.
    \end{enumerate}
    
    Having computed the type of each one-element subset of children of $uv$, we proceed to combine these to obtain, again, either a good set or a type for the set $W$ of \emph{all} children of $uv$ to be one of the nine in \eqref{eq:all-types}.
    To this end, assume that we already determined the type $_{\alpha_1}(s_1)_{\beta_1}$ of a non-empty subset $X$ of children of $uv$, and the type $_{\alpha_2}(s_2)_{\beta_2}$ of another non-empty subset $Y$ of children of $uv$ with $X \cap Y = \emptyset$, and that both these types are among the nine types in \eqref{eq:all-types}.
    We shall consider the set $W = X \cup Y$ now.
    
    We set $S_W = S_X \cup S_Y$ and calculate the surplus
    \begin{multline*}
        \mathrm{sp}(\{w\}) = 9|S_w|-5(|V(G_w)|-2) = 9(|S_X|+|S_Y|)-5(|V(G_X)|-2+|V(G_Y)|-2)\\
        = 9|S_X|-5(|V(G_X)|-2) + 9|S_Y| - 5(|V(G_Y)|-2) = s_1+s_2,
    \end{multline*}
    i.e., the sum of the two surpluses.
    For the threats at $u$ and $v$, we have $\mathrm{th}_W(u) = \alpha_1 + \alpha_2$ and $\mathrm{th}_W(v) = \beta_1 + \beta_2$.
    Again, observe from the list \eqref{eq:all-types} of all types, that $\mathrm{sp}(W) \geq 10$ and thus $W$ is a good set, unless both, $X$ and $Y$ have type $_1(4)_1$.
    But in this case $W$ has type $_2(8)_2$, which is again on the list \eqref{eq:all-types}.
    
    \medskip
    
    Thus by the above, we either find a good set, or determine a $3$-clustered set $S_W$ for the set $W$ of \emph{all} children of $uv$, such that the corresponding type of $W$ is on the list \eqref{eq:all-types}.
    Finally, we should argue that we will eventually find a good set.
    To this end, observe that in each of the above cases, whenever we determine a type $_\alpha(s)_\beta$ of some set $W$ based on two already determined types $_{\alpha_1}(s_1)_{\beta_1}$ and $_{\alpha_2}(s_2)_{\beta_2}$, then the new type $_\alpha(s)_\beta$ is further right in the list \eqref{eq:all-types} than both $_{\alpha_1}(s_1)_{\beta_1}$ and $_{\alpha_2}(s_2)_{\beta_2}$.
    Thus, this procedure will eventually find a good set or we have determined the type of the set $W_0$ of all children of the base edge $u_0v_0$ and it is on the list \eqref{eq:all-types}.
    In the latter case $S = S_{W_0} \cup \{u_0\}$ is a $3$-clustered set and we have
    \[
        9|S| = 9|S_{W_0}|+9 = \mathrm{sp}(W_0) + 5(|V|-2) + 9 \geq 4 + 5|V| - 10 + 9 = 5|V|+3. 
    \]
    In particular $|S| \geq \frac59 |V|$, as desired.
\end{proof}

\subsection{Extension to larger $c$ and $k$}

Let us remark that the proof strategy for \cref{prop:x23-upper-bound} with types, surplus, and threats can be adjusted to each fixed $c$ and $k$, to certify lower bounds of the form $x_{k,c} \geq \frac{p}{q}$, and possibly also find matching upper bound examples.

For $k=2$, fixed $c$ and test threshold $p/q$, we determine (as in the proof above) for each edge $uv$ and set $W$ of children of $uv$ a corresponding $c$-clustered set $S_W$.
This determines a type, which includes the surplus $\mathrm{sp}(W) = q|S_W|-p(|V(G_W)|-2)$ and in general three (not just two) threats.
The threat at $u$ is the total size of components of $G[S_W]$ that contain a neighbor of $u$ \emph{but no} neighbor of $v$, the threat at $v$ is symmetrical, and the common threat at $u,v$ is the total size of components of $G[S_W]$ that contain a neighbor of $u$ \emph{and} a neighbor of $v$ (hence a vertex of $W$).
        
We start with the type corresponding to $W = \emptyset$, which has surplus~$0$ and each threat~$0$.
Then, we exhaustively combine two known types to a single type by \textbf{(1)} knowing the types for all children of $uw$ and all children of $vw$ and combining these to the type for the single children $W = \{w\}$ of $uv$, and \textbf{(2)} knowing the types of two disjoint subsets $X,Y$ of children of $uv$ and combining these to the type of $X \cup Y$.
If this, as in the above proof, generates only a \emph{finite} list of types (which implies that each surplus is non-negative) without cyclic dependencies, this proves that indeed $x_{2,c} \geq \frac{p}{q}$.
On the other hand, if we encounter a type with negative surplus, then tracing back the combinations, we obtain a particular $2$-tree $G$.
Linearly arranging copies of $G$, this could\footnote{it did, in all cases we considered} lead to a family of $2$-trees certifying that $x_{2,c} < \frac{p}{q}$.

We have implemented this strategy for $k=2$ and small $c$.
The obtained results suggest that the true value of $x_{2,c}$ depends on the parity of $c$ modulo $3$:
\[
    \renewcommand{\arraystretch}{1.3}
    \begin{tabularx}{\linewidth}{r||Z|W|Y|Z|W|Y|Z|W|Y|Z|W|Y|Z|W}        
        $c$ & $2$ & $3$ & $4$ & $5$ & $6$ & $7$ & $8$ & $9$ & $10$ & $11$ & $12$ & $13$ & $14$ & $15$ \\
        \hline
        $x_{2,c}$ &
        $\frac{1}{2}$ &
        $\frac{5}{9}$ &
        $\frac{8}{13}$ &
        $\frac{2}{3}$ &
        $\frac{9}{13}$ &
        $\frac{13}{18}$ &
        $\frac{3}{4}$ &
        $\frac{13}{17}$ &
        $\frac{18}{23}$ &
        $\frac{4}{5}$ &
        $\frac{17}{21}$ &
        $\frac{23}{28}$ &
        $\frac{5}{6}$ &
        $\frac{21}{25}$ \\[2pt]
    \end{tabularx}
\]

Finally, the same approach can work for larger $k$, but the computational effort increases.
In rooted $k$-trees we would have parent $k$-cliques $K$, a type would store a threat for each non-empty subset of $K$, and to determine the type of a single children $W = \{w\}$ of $K$, we would combine the known types of all children of the $k$-subsets of $K \cup \{w\}$ containing $w$.
As this approach determines $x_{k,c}$ only for singular values of $k$ and $c$, we did not embark upon this path.

\section{Other graph classes}
\label{sec:conclusions}

Let us briefly discuss other classes besides the class of treewidth-$k$ graphs.
In accordance with \eqref{eq:definition-x-kc}, for any graph class $\mathcal{G}$ let us define
\[
    x_{\mathcal{G},c} = \inf \{\frac{\alpha_c(G)}{|V(G)|} \colon G \in \mathcal{G}\}\text{.}
\]

\begin{lemma}\label{lem:lb}
    If a graph class $\mathcal{G}$ is closed under vertex-disjoint unions and $G \in \mathcal{G}$ is $k$-connected on $k+c$ vertices, then $x_{\mathcal{G},c} \leq \frac{c}{k+c}$.
\end{lemma}
\begin{proof}
     Take a $k$-connected $G\in \mathcal{G}$ on $k+c$ vertices.
     Whenever we remove $k-1$ vertices, we get a connected subgraph on $c+1$ vertices.
     Hence, $\alpha_c(G) \leq c = \frac{c}{k+c}|V(G)|$, and  taking vertex-disjoint unions of $G$, we conclude $x_{\mathcal{G},c} \leq \frac{c}{k+c}$.
\end{proof}

In fact, \cref{obs:upper-bound} giving $x_{k,c} \leq \frac{c}{k+c}$ is just a special case of \cref{lem:lb} applied to the class of all graphs of treewidth~$k$.

Looking at the class $\mathcal{P}$ of all planar graphs, we can also apply~\cref{lem:lb}.
For $c\in \{2,\ldots,6\}$ we can take a $4$-connected planar graph on $c+4$ vertices, which gives an upper bound of $x_{\mathcal{P},c} \leq \frac{c}{c+4}$.
From $c\geq 7$ on one can even choose a $5$-connected planar graph and get an upper bound of $x_{\mathcal{P},c} \leq \frac{c}{c+5}$.
These bounds are not tight in general.
For instance for the icosahedron graph $G$ (shown in the middle of \cref{fig:example-clustered-sets}) we have $\alpha_5(G) = 6$ and $\alpha_6(G) = 7$, and taking vertex-disjoint unions of it yields the better bounds $\alpha_{\mathcal{P},5} \leq \frac{5}{12}$ and $\alpha_{\mathcal{P},6} \leq \frac{1}{2}$.
For $c=2$, \cref{lem:lb} yields $\alpha_{\mathcal{P},2} \leq \frac{1}{3}$ by taking the octahedron (shown on the left of \cref{fig:example-clustered-sets}), which we conjecture to be best-possible.

\begin{conjecture}\label{conj:1/3}
    In every planar graph there is a set $S$ on at least a third of the vertices, such that each vertex in $S$ is adjacent to at most one other vertex in $S$, i.e., $\alpha_{\mathcal{P},2}=\frac{1}{3}$. 
\end{conjecture}

We remark that \cref{conj:1/3} is implied by the Albertson-Berman Conjecture~\cite{AB79} that every planar graph $G$ admits an induced forest on at least half of the vertices.
In fact, any such forest would contain a $2$-clustered set on at least $\frac23$ of its vertices (hence $\frac13$ of the vertices of $G$) by \cref{thm:main}\labelcref{enum:k1}.
Along the same lines, we get the best known lower bound by the acyclic $5$-coloring of Borodin~\cite{Bor79}, which implies the existence of an induced forest on at least $\frac25$ of the vertices.
Hence, $\alpha_{\mathcal{P},2} \geq \frac25 \cdot \frac23 = \frac{4}{15}$.

\medskip

Concerning $x_{\mathcal{G},c}$ asymptotically, recall that for any graph $G$ we have $\alpha_1(G) \leq \alpha_2(G) \leq \cdots$ and hence for any graph class $\mathcal{G}$ we have $x_{\mathcal{G},1} \leq x_{\mathcal{G},2} \leq \cdots \leq 1$.
Edwards and McDiarmid~\cite{EM94} define a graph class $\mathcal{G}$ to be \emph{fragmentable} if for every $\varepsilon > 0$ there exist integers $c,n_0$ such that each graph $G \in \mathcal{G}$ with $n \geq n_0$ vertices admits a $c$-clustered set of size at least $(1-\varepsilon)n$.
In other words, $\mathcal{G}$ is fragmentable if $x_{\mathcal{G},c} \to 1$ as $c \to \infty$.
Edwards and McDiarmid prove that any class with \emph{strongly sublinear separators}\footnote{There exists a fixed $\varepsilon < 1$ such that every $n$-vertex $G \in \mathcal{G}$ has a $S \subseteq V(G)$ of size $O(n^{\varepsilon})$ such that each component of $G-S$ has at most $\frac{n}{2}$ vertices.} is fragmentable~\cite{EM94}.
This includes planar graphs~\cite{LT79}, graphs of bounded orientable genus~$g$~\cite{GHT84}, proper minor-closed graph classes~\cite{AST90}, $k$-planar graphs~\cite{GB07}, and touching graphs of $d$-dimensional balls~\cite{MTTV97}.

\begin{observation}
    A graph class $\mathcal{G}$ is fragmentable if and only if $\lim_{c \to \infty} x_{\mathcal{G},c} = 1$.
\end{observation}

While we might be able to derive an explicit lower bound on $\alpha_c(G)$ for any $c$ and any $G \in \mathcal{G}$ from a proof that $\mathcal{G}$ is fragmentable, our results suggest finding large induced subgraphs of bounded treewidth in $G$.
For example, for any proper minor-closed class $\mathcal{G}$ there exists a constant $k$ such that any graph $G = (V,E)$ with $G \in \mathcal{G}$ admits a $2$-coloring of $V$ for which each monochromatic induced subgraph has treewidth at most $k$~\cite{DDOSRSV04}.
This gives $\alpha_c(G)/|V| \geq \frac12 \cdot x_{k,c}$ for every $c \geq 0$.

\subsection*{Acknowledgments}

The authors thank David Wood for pointing us to~\cite{Wo} and Thomas Bl\"asius for mentioning the connection to $k$-vertex separators and $\ell$-component order connected sets. Part of this work was conducted at the 11th Annual Workshop on Geometry and Graphs held at the Bellairs Research Institute in March 2024. We are grateful to the organizers and participants for providing an excellent research environment.
KK was supported by he grant of The Natural Science Foundation of Hebei Province (project No. A2023205045),
 PID2022-137283NB-C22 funded by MICIU/AEI/10.13039/501100011033 and ERDF/EU, Severo Ochoa and María de Maeztu Program for Centers and Units of Excellence in R\&D (CEX2020-001084-M), ANR project MIMETIQUE: ANR-25-CE48-4089-01. TU was supported by the Deutsche Forschungsgemeinschaft (DFG, German Research Foundation) – 520723789.

\bibliographystyle{plainurl}
\bibliography{lit}

@article{BK18,
  title = {Bounded monochromatic components for random graphs},
  author = {Broutin, Nicolas and Kang, Ross J.},
  journal = {Journal of Combinatorics},
  volume = {9},
  number = {3},
  year = {2018},
  doi = {10.4310/JOC.2018.v9.n3.a1},
  pages = {411--446}
}

@inproceedings{BFNNPZ23,
  author = {Baguley, Samuel and Friedrich, Tobias and Neumann, Aneta and Neumann, Frank and Pappik, Marcus and Zeif, Ziena},
  title = {Fixed Parameter Multi-Objective Evolutionary Algorithms for the W-Separator Problem},
  year = {2023},
  publisher = {Association for Computing Machinery},
  address = {New York, NY, USA},
  NNurl = {https://doi.org/10.1145/3583131.3590501},
  doi = {10.1145/3583131.3590501},
  booktitle = {Proceedings of the Genetic and Evolutionary Computation Conference},
  pages = {1537–-1545},
  location = {Lisbon, Portugal},
  series = {GECCO '23}
}

@Book{BB72,
 Author = {Bertele, Umberto and Brioschi, Francesco},
 Title = {Nonserial dynamic programming},
 FSeries = {Mathematics in Science and Engineering},
 Series = {Math. Sci. Eng.},
 Volume = {91},
 Year = {1972},
 Publisher = {Elsevier, Amsterdam},
 Language = {English},
 zbMATH = {3386445},
 Zbl = {0244.49007}
}

@Article{Hal76,
 Author = {Halin, Rudolf},
 Title = {S-functions for graphs},
 FJournal = {Journal of Geometry},
 Journal = {J. Geom.},
 ISSN = {0047-2468},
 Volume = {8},
 Pages = {171--186},
 Year = {1976},
 Language = {English},
 DOI = {10.1007/BF01917434},
 zbMATH = {3529884},
 Zbl = {0339.05108}
}

@Article{ChPe,
 Author = {Chappell, Glenn G. and Pelsmajer, Michael J.},
 Title = {Maximum induced forests in graphs of bounded treewidth},
 FJournal = {The Electronic Journal of Combinatorics},
 Journal = {Electron. J. Comb.},
 Volume = {20},
 Issue = {4},
 Pages = {\#P8},
 Year = {2013},
 doi = {10.37236/3826},
 Nurl = {www.combinatorics.org/ojs/index.php/eljc/article/view/v20i4p8},
}

@article{Wo,
 author = {Dvo{\v{r}}{\'a}k, Zden{\v{e}}k and Wood, David R.},
 title = {Product structure of graph classes with strongly sublinear separators},
 fjournal = {Innovations in Graph Theory},
 journal = {Innov. Graph Theory},
 issn = {3050-743X},
 volume = {2},
 pages = {191--222},
 year = {2025},
 language = {English},
 doi = {10.5802/igt.10},
 zbMATH = {8099014},
 Zbl = {1573.05319}
}

@incollection{AB79,
    author = {Albertson, Michael O. and 
        Berman, David M.},
    title = {A conjecture on planar graphs},
    booktitle = {Graph Theory and Related Topics},
editor = {Bondy, J.A. and Murty, U.S.R.},
    publisher = {Academic Press New York},
    year = {1979}
}

@article{Bor79,
    author = {Oleg V. Borodin},
    title = {On acyclic colorings of planar graphs},
    journal = {Discrete Mathematics},
    volume = {25},
    number = {3},
    pages = {211--236},
    year = {1979},
    doi = {10.1016/0012-365X(79)90077-3},
    NNurl = {https://www.sciencedirect.com/science/article/pii/0012365X79900773},
}

@article{EM94,
    author = {Edwards, Keith and McDiarmid, Colin},
    title = {New upper bounds on harmonious colorings},
    journal = {Journal of Graph Theory},
    volume = {18},
    number = {3},
    pages = {257--267},
    doi = {10.1002/jgt.3190180305},
    NNurl = {https://onlinelibrary.wiley.com/doi/abs/10.1002/jgt.3190180305},
    year = {1994}
}

@article{LT79,
    author = {Lipton, Richard J. and Tarjan, Robert Endre},
    title = {A Separator Theorem for Planar Graphs},
    journal = {SIAM Journal on Applied Mathematics},
    volume = {36},
    number = {2},
    pages = {177--189},
    year = {1979},
    doi = {10.1137/0136016},
    Nurl = {https://doi.org/10.1137/0136016},
}

@article{GHT84,
    author = {John R Gilbert and Joan P Hutchinson and Robert Endre Tarjan},
    title = {A separator theorem for graphs of bounded genus},
    journal = {Journal of Algorithms},
    volume = {5},
    number = {3},
    pages = {391--407},
    year = {1984},
    doi = {10.1016/0196-6774(84)90019-1},
    NNurl = {https://www.sciencedirect.com/science/article/pii/0196677484900191},
}

@article{AST90,
    title = {A separator theorem for nonplanar graphs},
    author = {Alon, Noga and Seymour, Paul and Thomas, Robin},
    journal = {Journal of the American Mathematical Society},
    volume = {3},
    number = {4},
    pages = {801--808},
    doi = {10.2307/1990903},
    NNurl = {https://doi.org/10.1090/S0894-0347-1990-1065053-0},
    year = {1990}
}

@article{GB07,
    title = {Algorithms for graphs embeddable with few crossings per edge},
    author = {Grigoriev, Alexander and Bodlaender, Hans L},
    journal = {Algorithmica},
    volume = {49},
    number = {1},
    pages = {1--11},
    year = {2007},
    doi = {10.1007/s00453-007-0010-x},
    NNurl = {https://doi.org/10.1007/s00453-007-0010-x},
    publisher = {Springer}
}

@article{MTTV97,
    author = {Miller, Gary L. and Teng, Shang-Hua and Thurston, William and Vavasis, Stephen A.},
    title = {Separators for Sphere-Packings and Nearest Neighbor Graphs},
    year = {1997},
    issue_date = {Jan. 1997},
    publisher = {Association for Computing Machinery},
    address = {New York, NY, USA},
    volume = {44},
    number = {1},
    NNurl = {https://doi.org/10.1145/256292.256294},
    doi = {10.1145/256292.256294},
    journal = {J. ACM},
    month = {jan},
    pages = {1-29},
}

@article{RS86,
    author = {Robertson, Neil and Seymour, Paul},
    title = {Graph minors. {II}. {A}lgorithmic aspects of tree-width},
    journal = {J. Algorithms},
    year = 1986,
    volume = 7,
    pages = {309--322},
    number = 3,
    doi = {10.1016/0196-6774(86)90023-4},
    NNurl = {https://www.sciencedirect.com/science/article/pii/0196677486900234}
}

@article{DDOSRSV04,
    author = {Matt DeVos and Guoli Ding and Bogdan Oporowski and Daniel P. Sanders and Bruce Reed and Paul Seymour and Dirk Vertigan},
    title = {Excluding any graph as a minor allows a low tree-width 2-coloring},
    journal = {Journal of Combinatorial Theory, Series B},
    volume = {91},
    number = {1},
    pages = {25--41},
    year = {2004},
    doi = {10.1016/j.jctb.2003.09.001},
    NNurl = {https://www.sciencedirect.com/science/article/pii/S0095895603001059},
}

@article{Lee19,
    author = {Lee, Euiwoong},
    title = {Partitioning a Graph into Small Pieces with Applications to Path Transversal},
    year = {2019},
    publisher = {Springer-Verlag},
    address = {Berlin, Heidelberg},
    volume = {177},
    number = {1–2},
    NNurl = {https://doi.org/10.1007/s10107-018-1255-7},
    doi = {10.1007/s10107-018-1255-7},
    journal = {Math. Program.},
    pages = {1-19},
    numpages = {19},
}

@InProceedings{KL17,
    author = {Mithilesh Kumar and Daniel Lokshtanov},
    title =	{{A $2lk$ Kernel for $l$-Component Order Connectivity}},
    booktitle =	{11th International Symposium on Parameterized and Exact Computation (IPEC 2016)},
    pages =	{20:1--20:14},
    series = {Leibniz International Proceedings in Informatics (LIPIcs)},
    year = {2017},
    volume = {63},
    editor = {Jiong Guo and Danny Hermelin},
    publisher =	{Schloss Dagstuhl--Leibniz-Zentrum fuer Informatik},
    address = {Dagstuhl, Germany},
    Nurl = {http://drops.dagstuhl.de/opus/volltexte/2017/6934},
    doi = {10.4230/LIPIcs.IPEC.2016.20},  
}

@article{Yan81,
    author = {Yannakakis, Mihalis},
    title = {Node-Deletion Problems on Bipartite Graphs},
    journal = {SIAM Journal on Computing},
    volume = {10},
    number = {2},
    pages = {310--327},
    year = {1981},
    doi = {10.1137/0210022},
    Nurl = {https://doi.org/10.1137/0210022},
}

@article{BKKS11,
    author = {Bo\v{s}tjan Bre\v{s}ar and Franti\v{s}ek Kardo\v{s} and J\'{a}n Katreni\v{c} and Gabriel Semani\v{s}in},
    title = {Minimum k-path vertex cover},
    journal = {Discrete Applied Mathematics},
    volume = {159},
    number = {12},
    pages = {1189--1195},
    year = {2011},
    doi = {10.1016/j.dam.2011.04.008},
    NNurl = {https://www.sciencedirect.com/science/article/pii/S0166218X11001387},
}

\end{document}